\documentclass[a4paper,10pt]{amsart}
\pdfoutput=1
\def\definetac{\newif\iftac}

\ifx%
  \tactrue\undefined%
  \definetac%
  \ifx%
    \state\undefined\tacfalse%
  \else%
    \tactrue%
  \fi%
\fi%

\iftac\else\usepackage{amsthm}\fi

\usepackage{
   amssymb
  ,lmodern
  ,amsmath
  ,stmaryrd
  ,mathrsfs
  ,enumitem
  ,xcolor
  ,datetime
  ,leftidx
  ,marginnote
  ,mathtools
  ,tikz
  ,url
  ,newclude
  ,eucal
  ,xspace
}

\usetikzlibrary{calc}

\definecolor{darkgreen}{rgb}{0,0.45,0}

\usepackage[cmtip,all,2cell,pdf]{xy}\UseAllTwocells

\usepackage[%
   hyperfootnotes=true
  ,colorlinks%
  ,citecolor=darkgreen%
  ,linkcolor=darkgreen%
  ]{hyperref}

\usepackage[color=gray!20]{todonotes}

\usepackage{xcolor}
\usepackage{fixme}

\usepackage[%
 margin=1in
,includehead
,includefoot
,headheight=0.5in
,headsep=0.2in
,footskip=3em
,heightrounded
,left=5cm
,right=3cm
,top=4.5cm
,bottom=6cm%
,asymmetric]{geometry}

\usepackage[utf8]{inputenc}

\usepackage{tikz-cd}
\usetikzlibrary{positioning}
\usepackage{xparse}

\ExplSyntaxOn
\NewDocumentCommand{\makeabbrev}{mmm}
 {
  \yoruk_makeabbrev:nnn { #1 } { #2 } { #3 }
 }

\cs_new_protected:Npn \yoruk_makeabbrev:nnn #1 #2 #3
 {
  \clist_map_inline:nn { #3 }
   {
    \cs_new_protected:cpn { #2 } { #1 { ##1 } }
   }
 }
\ExplSyntaxOff

\makeabbrev{\mathbf}{b#1}{b,c,d,e,g,h,i,j,k,l,m,n,o,p,q,r,t,u,v,w,x,y,z,%
              A,B,C,D,E,G,H,I,J,K,L,M,N,O,P,Q,R,T,U,V,W,X,Y,Z}

\makeabbrev{\boldsymbol}{bs#1}{%
    a,b,c,d,e,g,h,i,j,k,l,m,n,o,p,q,r,s,t,u,v,w,x,y,z,%
    A,B,C,D,E,G,H,I,J,K,L,M,N,O,P,Q,R,S,T,U,V,W,X,Y,Z}

\makeabbrev{\mathsf}{sf#1}{a,b,c,d,e,f,g,h,i,j,k,l,m,n,o,p,q,r,s,t,u,v,w,x,y,z,%
                           A,B,C,D,E,F,G,H,I,J,K,L,M,N,O,P,Q,R,S,T,U,V,W,X,Y,Z}

\makeabbrev{\mathfrak}{f#1}{a,b,c,d,e,f,g,h,j,k,l,m,n,o,p,q,r,s,t,u,v,w,x,y,z,%
                             A,B,C,D,E,F,G,H,I,J,K,L,M,N,O,P,Q,R,S,T,U,V,W,X,Y,Z}

\makeabbrev{\mathcal}{c#1}{A,B,C,D,E,F,G,H,I,J,K,L,M,N,O,P,Q,R,S,T,U,V,W,X,Y,Z}
\makeabbrev{\mathbf}{l#1}{A,B,C,D,E,F,G,H,I,J,K,L,M,N,O,P,Q,R,S,T,U,V,W,X,Y,Z}
\makeabbrev{\mathbb}{s#1}{A,B,C,D,E,F,G,H,I,J,K,L,M,N,O,P,Q,R,S,T,U,V,W,X,Y,Z}

\makeatletter
\newif\ifhyperref
\@ifpackageloaded{hyperref}{\hyperreftrue}{\hyperreffalse}
\iftac
  \let\your@state\state
  \def\state#1{\gdef\currthmtype{#1}\your@state{#1}}
  \let\your@staterm\staterm
  \def\staterm#1{\gdef\currthmtype{#1}\your@staterm{#1}}
  \let\defthm\newtheorem
  \def\currthmtype{}
  \ifhyperref
    \def\autoref#1{\ref*{label@name@#1}~\ref{#1}}
  \else
    \def\autoref#1{\ref{label@name@#1}~\ref{#1}}
  \fi
  \AtBeginDocument{%
    \let\old@label\label%
    \def\label#1{%
      {\let\your@currentlabel\@currentlabel%
        \edef\@currentlabel{\currthmtype}%
        \old@label{label@name@#1}}%
      \old@label{#1}}
  }
\else
  \ifhyperref
    \def\defthm#1#2{%
      \newtheorem{#1}{#2}[section]%
      \expandafter\def\csname #1autorefname\endcsname{#2}%
      \expandafter\let\csname c@#1\endcsname\c@thm}
  \else
    \def\defthm#1#2{\newtheorem{#1}[thm]{#2}}
    \ifx\SK@label\undefined\let\SK@label\label\fi
    \let\old@label\label
    \let\your@thm\@thm
    \def\@thm#1#2#3{\gdef\currthmtype{#3}\your@thm{#1}{#2}{#3}}
    \def\currthmtype{}
    \def\label#1{{\let\your@currentlabel\@currentlabel\def\@currentlabel%
        {\currthmtype~\your@currentlabel}%
        \SK@label{#1@}}\old@label{#1}}
    \def\autoref#1{\ref{#1@}}
  \fi
\fi
\makeatother
\numberwithin{equation}{section}
\newtheorem{thm}{Theorem}[section]
\iftac\theoremstyle{plain}\else\theoremstyle{definition}\fi
\defthm{cor}{Corollary}
\defthm{prop}{Proposition}
\defthm{lem}{Lemma}
\defthm{sch}{Scholium}
\defthm{assume}{Assumption}
\defthm{claim}{Claim}
\defthm{conj}{Conjecture}
\defthm{hyp}{Hypothesis}
\defthm{fact}{Fact}
\defthm{quest}{Question}
\defthm{defn}{Definition}
\defthm{defn-prop}{Definition-Proposition}
\defthm{notn}{Notation}
\iftac\theoremstyle{plain}\else\theoremstyle{remark}\fi
\defthm{rmk}{Remark}
\defthm{eg}{Example}
\defthm{egs}{Examples}
\defthm{ex}{Exercise}
\defthm{ceg}{Counterexample}
\defthm{con}{Construction}
\defthm{warn}{Warning}
\defthm{digr}{Digression}
\defthm{per}{Perspective}
\defthm{disc}{Discussion}
\defthm{term}{Terminology}
\defthm{heur}{Heuristics}

\DeclareSymbolFont{bbold}{U}{bbold}{m}{n}
\usepackage{eucal}
  \DeclareSymbolFontAlphabet{\mathbbb}{bbold}

\DeclareMathOperator*{\colim}{colim}

\newcommand{\bDelta}{\boldsymbol{\Delta}}
\newcommand{\sSet}{\textbf{sSet}}

\newlength{\seplen}
\newlength{\sepwid}
\def\firstblank{\,\rule{\seplen}{\sepwid}\,}

\providecommand{\abbrv}[1]{#1.\@\xspace}
  \providecommand{\ie}{\abbrv{i.e}}

\def\Ran{\text{Ran}}

\def\Lan{\text{Lan}}

\def\ic{$\infty$-category\@\xspace}
\def\ics{$\infty$-categories\@\xspace}
\def\iCat{\cC\mathrm{at}_\infty}
\def\N{\text{N}}
\def\kan{\cK\mathrm{an}}
\let\xto\xrightarrow

\def\op{\text{o}}
\def\Tw{\text{Tw}}
\def\Map{\text{Map}}

\DeclareDocumentEnvironment{enumtag}{m }{\begin{enumerate}[label=\textsc{#1}\oldstylenums{\arabic*}),ref=\textsc{#1}\oldstylenums{\arabic*}]}{\end{enumerate}}

\usepackage{hyphenat}
\setlist[1]{itemsep=0pt}

\title{A Fubini rule for $\infty$-coends}
\author{Fosco Loregian}
\date{\today}
\address{
	Fosco \textsc{Loregian}\newline
	Max Planck Institute for Mathematics\newline
	Vivatsgasse 7, 53111 Bonn --- Germany\newline
	\href{mailto:flore@mpim-bonn.mpg.de}{\sf flore@mpim-bonn.mpg.de}
}

\begin{document}
\begin{abstract}
	We prove a Fubini rule for $\infty$-co/ends of $\infty$-functors $F : \cC^\op\times\cC\to \cD$. This allows to lay down ``integration rules'', similar to those in classical co/end calculus, also in the setting of \ics.
\end{abstract} 
\maketitle
\tableofcontents
\section{Introduction}
In \cite[§5.2.1]{LurieHA} (we shorten the reference to this source to ``HA'' from now on, and similarly we call simply ``T'' the reference \cite{HTT}) the author introduces the definition of \emph{twisted arrow \ic} of an \ic; in \cite{gepner2015lax} this paves the way to the definition of co/end for a $\infty$-functor $F : \cC^\op\times\cC\to \cD$. Here we briefly recall how this construction works.
\begin{defn}
Let $\epsilon \colon \bDelta \to \bDelta$ be the functor $[n] \mapsto   [n] \star [n]^{\op}$. The \emph{edgewise subdivision} $\text{esd}(X)$ of a simplicial set $S$ is defined to be the composite $\epsilon^{*}S$. If $\cC$ is an \ic, we define $\Tw(\cC)$ to be the simplicial set $\epsilon^*\cC$. The $n$-simplices of $\Tw(\cC)$ are, in particular, determined as
\[
\Tw(\cC)_n \cong \textbf{sSet}([n], \Tw(\cC)) = \textbf{sSet}([n]\star [n]^\op,\cC).
\]
\end{defn}
\begin{rmk}
In dimension $0$ and $1$, the $n$-simplices of $\Tw(\cC)$ correspond respectively to edges $f$ of $\cC$ and to commutative squares
\[
\xymatrix{
	\ar[d]_f & \ar[l]^s \ar[d]^g\\
	\ar[r]_t & 
}
\] 
The canonical natural transformations given by the embedding of $[n], [n]^\op$ in the join entail that there is a well-defined \emph{projection map} $\Sigma_{\cC} : \Tw(\cC)\to \cC^\op\times\cC$. Note that from HA.5.2.1.11 we deduce that $\Sigma_{\cC}$ is the right fibration HA.5.2.1.3 (this entails that if $\cC$ is an \ic, then $\Tw(\cC)$ is also an \ic) classified by $\Map : \cC^\op\times\cC\to \cS$.

If the \ic $\cC$ is of the form $\text{N}(A)$ for some category $A$, $\Tw(\cC)$ corresponds to the nerve of the classical twisted arrow category of $A$, as defined in \cite[IX.6.3]{McL}.
\end{rmk}
\begin{defn}
Let $F : \cC^\op\times\cC\to \cD$ be a functor; when it exists, the \emph{end} of $F$ is the limit
\[
\int_C F := \lim_{\Tw(\cC)} (F\cdot\Sigma)
\]
Dually, when it exists, the \emph{coend} of $F$ is the colimit 
\[
\int^C F := \colim_{\Tw(\cC)} (F\cdot\Sigma)
\]
\end{defn}
It is clear that a sufficient condition for $\int^C F$ to exists is that $\cD$ is cocomplete, and dually a sufficient condition for $\int_C F$ to exist is that $\cD$ is complete.

\cite{gepner2015lax} employs this notation to prove [\emph{ibi}, Thm. 1.1] that 
\begin{thm}
Suppose $F \colon \cC \to \iCat$ is a functor of \ics,
\begin{enumtag}{lc}
\item if $\cE \to \cC$ is a cocartesian fibration associated to $F$. Then $\cE$ is the \emph{lax colimit}\footnote{The lax colimit of $F : \cC \to \iCat$ is defined by the coend
\[\int^C \cC_{C/} \times F(C).\]
Dually, the oplax colimit of $F$ is defined by the coend
\[\int^C \cC_{/C} \times F(C),\]
where in both cases the weights are the \emph{slice} \ics of T.1.2.9.2 and T.1.2.9.5.} of the functor $F$.
\item if $\cE \to \cC$ is a cartesian fibration associated to $F$. Then $\cE$ is the \emph{oplax colimit} of the functor $F$.
\end{enumtag}
\end{thm}
 \begin{lem}\label{adjunccc}
Let $\cC$ be a small \ic, and $\cD$ be a presentable \ic; then $\cD$ is tensored and cotensored over $\cS = \N(\kan)$. This entails that there is a two-variable adjunction
\[
\cD^\op \times \cD \xto{\Map} \cS
\qquad
\cS \times \cD \xto{\odot} \cD
\qquad
\cS^\op\times\cD \xto{\pitchfork}\cD
\]
such that
\[ \cD(X\odot D,D')\cong \cS(X, \cD(D,D'))\cong \cD(D, X\pitchfork D') \]
\end{lem}
From the existence of these isomorphisms it is clear that
\begin{align}
V\odot(W\odot D) \cong W\odot(V\odot D) \cong (V\times W)\odot D\label{iso:fst}\\
V\pitchfork(W\pitchfork D) \cong W\pitchfork(V\pitchfork D)\cong (V\times W)\pitchfork D\label{iso:snd}
\end{align}
\section{The Fubini rule}
\begin{lem}\label{end-is-functor}
Let $F : \cC^\op\times\cC\to \cD$ be a $\infty$-functor and $\cC,\cD$ \ics as in the assumptions of \autoref{adjunccc}. Then
\begin{itemize}
	\item $F\mapsto \int^C F$ is functorial, and it is a left adjoint;
	\item $F\mapsto \int_C F$ is functorial, and it is a right adjoint.
\end{itemize}
\end{lem}
\begin{proof}
We only prove the first statement for coends; the other one is dual.

Since $\int^C F = \colim_{\Tw(\cC)}(F\cdot \Sigma)$ acts on $F$ as a composition of $\infty$-functors, it is clearly functorial; then in the diagram
\[\textstyle
\int^C : 
\xymatrix@C=2cm{
[\cC^\op\times\cC,\cD] \ar@<5pt>[r]^-{\Sigma^*} &\ar@<5pt>@{.>}[l]^-{\Ran_\Sigma} \ar@{}[l]|-\perp[\Tw(\cC),\cD] \ar@<5pt>[r]^-{\colim_{\Tw(\cC)}} &\ar@<5pt>@{.>}[l]^-{c} \ar@{}[l]|-\perp\cD\\
}
\]
the composition $\int^C = \colim_{\Tw(\cC)} \cdot \Sigma^*$ is a left adjoint because it is a composition of left adjoints ($c = t^*$ is the constant functor inverse image of the terminal map $\Tw(\cC)\to *$). Dually, the left adjoint to the end functor $\int_C $ is given by $\Lan_\Sigma \cdot c(D)$.
\end{proof}
Loosely speaking, the Fubini rule for co/ends asserts that when the domain of a functor $F : \cA^\op\times\cA\to \cD$ results as a product $(\cC\times\cE)^\op\times(\cC\times\cE)$, then the co/ends of $F$ can be computed as ``iterated integrals''
\begin{gather}
\int^{(C,E)} F \cong \iint^{CE} F \cong \iint^{EC} F\label{coends}\\
\int_{(C,E)} F \cong \iint_{CE} F \cong \iint_{EC} F\label{ends}
\end{gather}
These identifications hide a slight abuse of notation, that is worth to make explicit in order to avoid confusion: thanks to \autoref{end-is-functor} the three objects of \eqref{coends} can be thought as images of $F$ along certain functors, and the Fubini rule asserts that they are linked by canonical isomorphisms; we can easily turn these functors into having the same type by means of the cartesian closed structure of $\sSet$:
\begin{gather}
 \iint^{CE} := [\cC^\op\times\cC\times\cE^\op\times\cE,\cD] \cong [\cC^\op\times\cC, [\cE^\op\times\cE,\cD]] \xto{[\cC^\op\times\cC,\int^E]} [\cC^\op\times\cC,\cD]\xto{\int^C} \cD\label{primo}\\
 \iint^{EC} := [\cC^\op\times\cC\times\cE^\op\times\cE,\cD] \cong [\cE^\op\times\cE, [\cC^\op\times\cC,\cD]] \xto{[\cE^\op\times\cE,\int^C]} [\cE^\op\times\cE,\cD]\xto{\int^E} \cD\label{secondo}\\
 \int^{(C,E)} := [\cC^\op\times\cC\times\cE^\op\times\cE,\cD] \cong [(\cC\times\cE)^\op\times(\cC\times\cE),\cD] \to \cD.\label{terzo}
 \end{gather} 
(of course, we can provide similar definitions for the iterated end functor).

Once that this has been clarified, we can deduce the isomorphisms \eqref{coends} and \eqref{ends} from the fact that the three functors $\iint^{CE},\iint^{EC},\int^{(C,E)}$ have right adjoints isomorphic to each other, and hence they must be isomorphic themselves.
\begin{thm}[Fubini rule for co/ends]\label{fubone}
Let $F : \cC^\op\times\cC\times\cE^\op\times\cE\to \cD$ be a $\infty$-functor. Then the $\infty$-functors $\iint^{CE},\iint^{EC},\int^{(C,E)}$ of \eqref{primo}, \eqref{secondo}, \eqref{terzo} are naturally isomorphic.
\end{thm}
In order to prove \ref{fubone} we need a preliminary observation characterizing the right adjoint to $\int^C : [\cC^\op\times\cC,\cD]\to \cD$.
\begin{lem}
The functor $R = \Ran_\Sigma(c(\_))$ acts ``cotensoring with mapping space'': more precisely, the functor $RD : \cC^\op\times\cC\to \cD$ is isomorphic to the functor
\[(C,C')\mapsto \Map_{\cC}(C,C')\pitchfork D\]
Dually, the functor $L = \Lan_\Sigma(c(\_))$ acts ``tensoring with mapping space'': more precisely, the functor $LD : \cC^\op\times\cC\to \cD$ is isomorphic to the functor
\[(C,C')\mapsto \Map_{\cC}(C,C')\odot D\]
\end{lem}
\begin{proof}
We only prove the first statement for coends; the other one is dual.

Being $c(D)$ the constant functor on $D\in\cD$, the pointwise formula for right Kan extensions (see \cite[6.4.9]{Cisino} for the fact that ``Ran are limits'') yields that the desired limit consists of cotensoring with the slice category $(C,C')/\Sigma$ regarded as a simplicial set (in the Kan-Quillen model structure); now, if we consider the diagram
\[
\xymatrix{
\Map_{\cC}(C,C') \ar[r]^\sim\ar[d]& (C,C')/\Sigma \ar[r]\ar[d]& \Tw(\cC) \ar[d]\\
\Delta^0 \ar[r]_\sim & (\cC^\op\times\cC)_{(C,C')/} \ar[r] & \cC^\op\times\cC
}
\]
expressing the fiber of $\Sigma$, \ie the mapping spaces $\Map_{\cC}(C,C')$, as suitable pullbacks, we can easily see that $(\cC^\op\times\cC)_{(C,C')/}$ is contractible in Kan-Quillen (it has an initial object), hence, in the $\infty$-category of spaces, the object $(C,C')/\Sigma$ exhibits the same universal property of $\Map_{\cC}(C,C')$. Since the functor $\firstblank\pitchfork D$ preserves Kan-Quillen weak equivalences, it turns out that
\[
\Ran_\Sigma(c(D)) \cong (C,C')/\Sigma \pitchfork D \cong \Map_{\cC}(C,C')\pitchfork D,
\]
and this concludes the proof.
\end{proof}
\begin{proof}[Proof of \autoref{fubone}]
The Fubini rule now follows from uniqueness of adjoints (T.5.2.1.3, T.5.2.1.4): in diagram
\[
\xymatrix@R=3mm{
\lambda F . \int^C\kern-.35em\int^E F \ar@{-|}[r] & \lambda D. \lambda CC'.\lambda EE'.\Map_\cC(C,C')\pitchfork \Big(\Map_\cE(E,E')\pitchfork D\Big)\ar@{=}[d]^\wr\\
\lambda F . \int^E\kern-.35em\int^C F \ar@{-|}[r] & \lambda D. \lambda EE'.\lambda CC'.\Map_\cE(E,E')\pitchfork \Big(\Map_\cC(C,C')\pitchfork D\Big)\ar@{=}[d]^\wr\\
 & \lambda D.\lambda CEC'E' \Big(\Map_\cC(C,C')\times\Map_\cE(E,E')\Big)\pitchfork D\ar@{=}[d]^\wr\\
\lambda F . \int^{(C,E)} F \ar@{-|}[r] & \Map_{\cC\times\cE}\big((C,E),(C',E')\big)\pitchfork D
}
\]
the vertical isomorphisms on the right are justified by \eqref{iso:snd}. A completely analogous argument, using \eqref{iso:fst} instead, and the left adjoints given by tensoring with the derived mapping space, gives the Fubini rule for \eqref{ends}.
\end{proof}

\paragraph{\bf Acknowledgements.} This document was partially written during the author's stay at MPIM. An initial draft of Lemma 2.3's proof has been tweaked by S. Ariotta in a private conversation.
\bibliography{allofthem}{}
\bibliographystyle{amsalpha}
\hrulefill
\end{document}